\documentclass[11pt]{amsart}

\usepackage{amssymb, amscd, latexsym}
\usepackage{amsmath}
\usepackage{amsthm}
\usepackage{graphicx}
\usepackage[all]{xy}
\usepackage{float}
\usepackage{mathrsfs}
\usepackage{pgf,tikz}
\usetikzlibrary{arrows}
\usepackage{array}
\usepackage{arydshln}
\usepackage{hyperref}

\hypersetup{
pdftitle={Research Paper},
pdfauthor={Ali Akbar Yazdan Pour},
pdfsubject={Betti Number and chordality},
pdfcreator={Ali Akbar Yazdan Pour},
colorlinks,
linkcolor=blue,
citecolor=magenta,
anchorcolor=red,
bookmarksopen,
urlcolor=red,
filecolor=red,
pdfpagetransition={Wipe}
}

\theoremstyle{plain}
\newtheorem{thm}{Theorem}[section]
\newtheorem{cor}[thm]{Corollary}
\newtheorem{lem}[thm]{Lemma}
\newtheorem{prop}[thm]{Proposition}

\theoremstyle{definition}
\newtheorem{defn}[thm]{Definition}
\newtheorem{ex}[thm]{Example}

\theoremstyle{remark}
\newtheorem{rem}{Remark}
\newtheorem{question}{Question}

\def\cocoa{{\hbox{\rm C\kern-.13em o\kern-.07em C\kern-.13em o\kern-.15em A}}}

\def\Ker{{\rm Ker}}
\def\reg{{\rm reg}}

\def\projdim{{\rm projdim}}

\def\C{{\mathcal C}}
\def\Index{\mathrm{index}}
\def\implies{\ifmmode\Rightarrow \else
        \unskip${}\Rightarrow{}$\ignorespaces\fi}
\def\Tor{\mathrm{Tor}}

\begin{document}

\title[reduction and chordality]{Stability of Betti numbers under reduction processes: towards  chordality of clutters}
\author[M. Bigdeli, A. A. Yazdan Pour, R. Zaare-Nahandi]{mina bigdeli, {ali akbar} {yazdan pour}, rashid zaare-nahandi}
\address{department of mathematics\\ institute for advanced studies in basic sciences (IASBS)\\ p.o.box 45195-1159 \\ zanjan, iran}
\email{m.bigdelie@iasbs.ac.ir, yazdan@iasbs.ac.ir, rashidzn@iasbs.ac.ir}


\begin{abstract}
For a given clutter $\mathcal{C}$, let $I:=I \left( \bar{\mathcal{C}} \right)$ be the circuit ideal in the polynomial ring $S$. In this paper, we show that the Betti numbers of $I$ and $I + \left( \textbf{x}_F \right)$ are the same in their non-linear strands, for some suitable $F \in \mathcal{C}$. Motivated by this result, we introduce a class of clutters that we call  chordal. This class, is a natural extension of the class of chordal graphs and has the nice property that the circuit ideal associated to any member of this class has a linear resolution over any field. Finally we compare this class with all known families of clutters which generalize the notion of chordality, and show that our class  contains several important previously defined  classes of chordal clutters. We also show that in comparison with others, this class is possibly the best approximation to the class of $d$-uniform clutters with linear resolution over any field.


\end{abstract}
\maketitle

\section*{introduction}
Square-free monomial ideals are in strong connection to topology and combinatorics. There are at least two approaches to investigate these ideals in terms of topology or combinatorics. One approach is to associate a simplicial complex to a given square-free monomial ideal $I$, whose  faces come from square-free monomials which do not belong to $I$. Another approach is to associate a clutter to $I$ whose circuits come from the minimal generators of $I$. The main goal in both cases is to obtain algebraic properties of $I$ via combinatorial or topological properties of associated objects. One of the highlighted result on this subject is Fr\"oberg's theorem.

R. Fr\"oberg in $1990$ showed that, the edge ideal of a graph $G$ has a linear resolution if and only if the complement graph $\bar{G}$ is chordal \cite{Fr}. In particular, for a square-free monomial ideal generated in degree $2$, the problem of having linear resolution depends only on the shape of the associated graph, and does not depend on the characteristic of the base field. This is not the case for square-free monomial ideals generated in degree $d>2$. The ideal associated to a triangulation of the projective plane, is a classical example of a square-free monomial ideal generated in degree $3$ whose resolution depends on the characteristic of the base field (see e.g. \cite[Section 4]{Katzman}). So it is too much to expect a combinatorial characterization (as in Fr\"oberg's theorem) for arbitrary square-free monomial ideals with linear resolution. However, it is reasonable to ask, if one may find such a characterization for (square-free) monomial ideals with linear resolution over any field. It is worth to say that, via Alexander duality, this problem is equivalent to characterization  of all simplicial complexes which are Cohen-Macaulay over any field (c.f. \cite[Theorem 3]{Eagon-Reiner}). As a partial result on this subject, in \cite{Emtander, HaVanTuyl, VanTuyl-Villarreal, Woodroofe}, the authors defined several generalizations of chordality to higher dimensions, and they showed that the ideal associated to their classes, have a linear resolution over any field. However, it is not so difficult to give a counterexample for the other direction. On the other hand, in \cite{ConnonFaridi}, the authors made the attempt to prove the other direction, by showing that every square-free monomial ideal with linear resolution over any field, comes from a chorded simplicial complex, where in that paper chorded is defined pretty technically. Yet, the authors show by an example that not any ideal admitting a linear resolution over any field, need to be chorded in their sense \cite[Example 7.2]{ConnonFaridi}.

The main aim of this paper is twofold. First we show that, for a given square-free monomial ideal $I$, we may add (remove) some generators to (from) $I$ in a way that, the corresponding non-linear strands do not change under this process. Then, motivated by this result, we introduce a class of clutters whose associated ideal of any member of this class has a linear resolution over any field. The advantage of this definition is that, this class contains other known families of clutters with this property, and at the moment, we don't know of any counterexample for the other direction (see Question~\ref{Characterization question}). 

The paper is organized as follows:
In the first section, we present the background material. This involves some preliminaries on graded modules and Betti numbers together with some basic notions of combinatorics. Then, in Section~\ref{Section Stability}, we state one of the main theorems of this paper (Theorem~\ref{main}). Indeed, with the required preparations, we show that, if $I$ is a square-free monomial ideal corresponded to a clutter $\C$ and $F \in \C$ is chosen appropriately, then the ideals $I+ \left( \textbf{x}_F \right)$ and $I$ share the same Betti numbers in their non-linear strands. In Section~\ref{Section Chordal}, we introduce the class $\mathfrak{C}_d$ of chordal clutters. It is shown that, for any member of $\mathfrak{C}_d$, the associated ideal has a linear resolution over any field. Then we show that, this class contains other families of chordal clutters as defined in \cite{Emtander, HaVanTuyl, VanTuyl-Villarreal, Woodroofe}. We close the paper by showing that, unlike in the graph case, it is not true that for arbitrary element of $\mathfrak{C}_d$, the associated ideal has linear quotients. A counterexample for  the last assertion, comes from  a triangulation of the dunce hat.


\section{preliminaries}

\subsection*{Algebraic backgrounds}

Throughout this paper, $S=K[x_1, \ldots, x_n]$ denotes the polynomial ring over a field $K$ with the standard grading (i.e. $\deg(x_i)=1$). Let $M \neq 0$ be a finitely generated graded $S$-module and
$$
 \cdots \to F_2 \to F_1 \to F_0 \to M \to 0
$$
a graded minimal free resolution of $M$ with $F_i = \oplus_j S(-j)^{\beta^K_{i,j}}$ for all $i$.

The numbers $\beta_{i,j}^K(M) = \dim_K \mbox{Tor}^S_i(K,M)_j$ are called the \textit{graded Betti numbers}
of $M$ and
\begin{align*}
\projdim (M) &= \sup\{i \colon \quad \mbox{Tor}^S_i(M,K) \neq 0\} \\
& = \max\{i \colon \quad  \beta^K_{i, i+j}(M) \neq 0 \text{ for some } j \}
\end{align*}
is called the \textit{projective dimension} of $M$. For simplicity, in this paper, we fix a field $K$ and we write simply $\beta_{i,j}$ instead of $\beta_{i,j}^K$.

The \textit{Castelnuovo-Mumford regularity} of $M \neq 0$, $\mathrm{reg}(M)$, is given by
$$
\mbox{reg}(M) = \sup\{j - i : \quad \beta_{i,j}(M) \neq 0\}.
$$
The \textit{initial degree} of $M$, $\mathrm{indeg}(M)$, is given by
$$
\mbox{indeg}(M) = \inf\{i : \quad M_i \neq 0\}.
$$
We say that a finitely generated graded $S$-module $M \neq 0$ has a \textit{$d$-linear resolution}, if its regularity is equal to $d = \mbox{indeg}(M)$. An important class of graded modules with linear resolution, is the class of modules which are generated in the same degree and have linear quotients \cite{HHBook}. Recall that $M$ is said to have \textit{linear quotients}, if $M$ has an ordered set of minimal generators $\left\{ m_1, \ldots, m_r \right\}$ such that the colon ideals $\left( m_1, \ldots, m_{i-1} \right) \colon m_i$ are generated by linear forms, for $i= 2, \ldots, r$.

In this paper, we concentrate on homogeneous ideals of $S$, in particular monomial ideals. 
An invariant of a homogeneous ideal in the polynomial ring, is the Green-Lazarsfeld index which measures the number of linear steps in the graded minimal free resolution of an ideal. The ideal $I \subset S$ is called \textit{$r$ steps linear}, if $I$ has a linear resolution up to homological degree $r$, in other words, if $\beta_{i,i+j}(I)=0$ for all pairs $(i,j)$ with  $0\leq i\leq r$ and $j> \mbox{indeg}(I)$. The number
\[
\mathrm{index}(I)=\sup\{r\colon \quad I \text{ is } r \text{ steps linear} \} + 1
\]
is called the  \textit{index} of $I$. In particular, $I$ has a linear resolution if and only if $\Index(I)=\infty$.

\subsection*{Clutters and circuit ideals}  Let us introduce some notations and terminologies which concerns about combinatorial commutative algebra. Let  $\left[ n \right]= \left\{ 1, \ldots, n \right\}$.

\begin{defn}[Clutter] \label{SC} 
A \textit{clutter} $\C$ with vertex set $[n]$ is a collection of subsets of $[n]$, called \textit{circuits} of $\C$, such that if $F_1$ and $F_2$ are distinct circuits, then $F_1 \nsubseteq F_2$.
A \textit{$d$-circuit} is a circuit consisting of exactly $d$ vertices, and a clutter is called \textit{$d$-uniform}, if every circuit has $d$ vertices.
A $(d-1)$-subset $e \subset [n]$ is called a \textit{submaximal circuit} of $\C$, if there exists $F \in \C$ such that $e \subset F$. The set of all submaximal circuits of $\C$ is denoted by ${\rm SC}(\C)$.
\end{defn}

Let $\C$ be a clutter with vertex set $[n]$. For a subset $W \subseteq [n]$, the \textit{induced subclutter} on $W$ is denoted by $\C|_W$ and is defined as follows:
$$\C|_W = \left\{ F \in \C \colon \quad F \subseteq W \right\}.$$
Also, for a non-empty clutter $\C$ with vertex set $[n]$, we define the ideal $I \left( \C \right)$, as follows:
$$I(\C) = \left(  \textbf{x}_T \colon \quad T \in \C \right),$$
where $\textbf{x}_T=x_{i_1}\cdots x_{i_t}$ for $T=\{i_1,\ldots,i_t\}$, and we define $I(\varnothing) = 0$.

Let $n$, $d$ be positive integers. For $n \geq d$, we define $\C_{n,d}$, the \textit{complete $d$-uniform clutter} on $[n]$, as follows:
$$\C_{n,d} = \left\{ F \subset [n] \colon \quad |F|=d \right\}.$$
In the case that $n<d$, we let $\C_{n,d}$ be some isolated points. It is well-known that, for $n \geq d$ the ideal $I \left( \C_{n,d} \right)$ has a $d$-linear resolution (see e.g. \cite[Example 2.12]{MNYZ}). 

If $\C$ is a $d$-uniform clutter on $[n]$, we define $\bar{\C}$, the \textit{complement} of $\C$, to be
\begin{equation*}
\bar{\C} = \C_{n,d} \setminus \C = \{F \subset [n] \colon \quad |F|=d, \,F \notin \C\}.
\end{equation*}

Frequently in this paper, we take a $d$-uniform clutter $\C \neq \C_{n,d}$ with vertex set $[n]$ and consider the square-free monomial ideal $I=I(\bar{\C})$ in the polynomial ring $S=K[x_1, \ldots, x_n]$. The ideal $I= I \left( \bar{\C} \right)$ is called the \textit{circuit ideal} of $\C$.

\begin{defn}
Let $\C$ be a $d$-uniform clutter on $[n]$. A subset $V \subset [n]$ is called a \textit{clique} in $\C$, if all $d$-subsets of $V$ belong to $\C$. Note that a subset of $[n]$ with less than $d$ elements is supposed to be a clique.
\end{defn}

For any $(d-1)$-subset $e$ of $[n]$, let
\begin{equation*}
{N}_{\C} \left[ e \right] = e \cup \lbrace c \in \left[ n \right] \colon \quad e \cup \lbrace c \rbrace \in \mathcal{C} \rbrace.
\end{equation*}
We call ${N}_{\C} \left[ e \right]$ the \textit{closed neighborhood} of $e$ in $\mathcal{C}$. We say that $e$ is \textit{simplicial} in $\C$, if ${N}_{\C} \left[ e \right]$ is a clique in $\mathcal{C}$. Let us denote by $\mathrm{Simp} \left( \mathcal{C} \right)$, the set of all simplicial elements of $\mathcal{C}$. More generally, for a subset $A \subset [n]$ with $|A| < d$, let 
\[
N_\C[A] = A \cup \left\{ c \in [n] \colon \quad A \cup \{c\} \subseteq F, \text{ for some } F \in \C \right\}.
\]

\begin{ex}
Figure~\ref{1} displays a $3$-uniform clutter $\mathcal{C}$ whose circuits are
$$
 \{ 1,2,3 \}, \{ 1,2,4 \}, \{ 1,3,4 \}, \{ 2,3,4 \}, \{2,3,6 \},
 \{ 2,5,6 \}, \{ 2,5,7 \}, \{2,6,7\}, \{ 5,6,7 \}.
 $$
In this clutter, $\left\{ 2,3 \right\}$ and $\left\{ 2,6 \right\}$ are not simplicial in $\C$, but all the other submaximal circuits of the set $\{1, \ldots, 7\}$ are simplicial in $\C$.
\begin{figure}[ht!]
\begin{tikzpicture}[line cap=round,line join=round,>=triangle 45,x=.83cm,y=.83cm]
\clip(4.18,2.99) rectangle (10.87,6.71);
\fill[fill=black,fill opacity=0.1] (6.32,5.72) -- (7,3.64) -- (5.28,4) -- cycle;
\fill[fill=black,fill opacity=0.1] (6.32,5.72) -- (7.66,4.38) -- (7,3.64) -- cycle;
\fill[fill=black,fill opacity=0.1] (7.66,4.38) -- (8.78,5.81) -- (8.33,3.66) -- cycle;
\fill[fill=black,fill opacity=0.1] (8.78,5.81) -- (10,4) -- (8.33,3.66) -- cycle;
\fill[fill=black,fill opacity=0.1] (7,3.64) -- (7.66,4.38) -- (8.33,3.66) -- cycle;
\draw (6.32,5.72)-- (7,3.64);
\draw (7,3.64)-- (5.28,4);
\draw (5.28,4)-- (6.32,5.72);
\draw  (6.32,5.72)-- (7.66,4.38);
\draw (7.66,4.38)-- (7,3.64);
\draw (7,3.64)-- (6.32,5.72);
\draw [dash pattern=on 2pt off 2pt] (5.28,4)-- (7.66,4.38);
\draw (6.05,6.31) node[anchor=north west] {\begin{scriptsize}1\end{scriptsize}};
\draw (7.4,4.31) node[anchor=north west] {\begin{scriptsize}2\end{scriptsize}};
\draw (6.73,3.63) node[anchor=north west] {\begin{scriptsize}3\end{scriptsize}};
\draw (4.8,4.33) node[anchor=north west] {\begin{scriptsize}4\end{scriptsize}};
\draw (7.66,4.38)-- (8.78,5.81);
\draw (8.78,5.81)-- (8.33,3.66);
\draw (8.33,3.66)-- (7.66,4.38);
\draw (8.78,5.81)-- (10,4);
\draw (10,4)-- (8.33,3.66);
\draw (8.33,3.66)-- (8.78,5.81);
\draw [dash pattern=on 2pt off 2pt] (7.66,4.38)-- (10,4);
\draw (7,3.64)-- (7.66,4.38);
\draw (7.66,4.38)-- (8.33,3.66);
\draw (8.33,3.66)-- (7,3.64);
\draw (8.47,6.35) node[anchor=north west] {\begin{scriptsize}5\end{scriptsize}};
\draw (8.12,3.67) node[anchor=north west] {\begin{scriptsize}6\end{scriptsize}};
\draw (9.93,4.29) node[anchor=north west] {\begin{scriptsize}7\end{scriptsize}};
\begin{scriptsize}
\fill [color=black] (6.32,5.72) circle (1.3pt);
\fill [color=black] (7,3.64) circle (1.3pt);
\fill [color=black] (5.28,4) circle (1.3pt);
\fill [color=black] (7.66,4.38) circle (1.3pt);
\fill [color=black] (8.78,5.81) circle (1.3pt);
\fill [color=black] (8.33,3.66) circle (1.3pt);
\fill [color=black] (10,4) circle (1.3pt);
\end{scriptsize}
\end{tikzpicture}
\caption{A $3$-uniform clutter}
\label{1}
\end{figure}
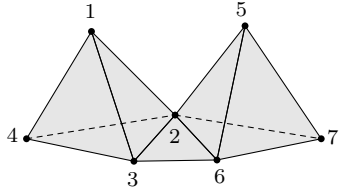
\end{ex}

\begin{defn}
Let $\mathcal{C}$ be a $d$-uniform clutter and let $e$ be a $(d-1)$-subset of $[n]$.  By $\mathcal{C} \setminus e$ we mean the $d$-uniform clutter 
$$ \left\{ F \colon \quad F \in \mathcal{C}, \ e \nsubseteq F \right\}.$$
It is called the \textit{deletion} of $e$ from $\mathcal{C}$. In the case that $e$ is not a submaximal circuit of $\C$, we have $\mathcal{C} \setminus e = \C$.
\end{defn}

\section{stability of non-linear strands under removing simplicial elements} \label{Section Stability}

The goal of this section is to show that deletion of simplicial elements of a $d$-uniform clutter does not change the Betti numbers of the corresponding non-linear strands. Indeed, the main goal of this section is to prove the following theorem.

\begin{thm} \label{main}
Let $\mathcal{C}$ be a $d$-uniform clutter and $e \in \mathrm{Simp}(\C)$ a simplicial element of $\mathcal{C}$. Let $A \subseteq \left\{ F \in \C \colon \quad e \subset F \right\}$  and  $\mathcal{D}=\mathcal{C} \setminus A$. Then,
\begin{equation*}
 \beta_{i,i+j} \left( I \left( \bar{\mathcal{C}} \right) \right) = \beta_{i,i+j} \left( I \left( \bar{\mathcal{D}} \right) \right) ,
\end{equation*}
for all $i$ and all $j>d$. Consequently,
\begin{itemize}
	\item[\rm (a)] $\reg \left( I \left( \bar{\mathcal{C}} \right)\right) = \reg \left( I \left( \bar{\mathcal{D}} \right) \right)$;  
	\item[\rm (b)] $\Index \left( I \left( \bar{\mathcal{C}} \right) \right) = \Index \left( I \left( \bar{\mathcal{D}} \right) \right)$;
	\item[\rm (c)] $\projdim \left( I \left( \bar{\mathcal{C}} \right) \right) \leq \projdim \left( I\left( \bar{\mathcal{D}} \right) \right)$.
\end{itemize}
\end{thm}

To  prove Theorem~\ref{main}, we need some preparations and we consider more generally graded ideals. In the next sections, we will introduce several applications of this theorem. Among other applications, Theorem~\ref{main} recovers and extends \cite[Theorem 2.7]{MYZ2} and \cite[Theorem 3.7]{MNYZ}.

\begin{lem} \label{ll}
Let $I, L\subset S$ be   graded ideals generated in degree $d$, and suppose that $L$ has a $d$-linear resolution. Then the following statements are equivalent:
\begin{itemize}
\item[(a)] the natural map $\alpha_{i,j}:\mathrm{Tor}_i(I,K)_{i+j} \to \mathrm{Tor}_i(I+L, K)_{i+j}$ induced by the inclusion $I \hookrightarrow I+L$ is an isomorphism for all $i$ and all  $j>d$;
\item[(b)] $\mathrm{reg} \left( I \cap L \right) \leq d+1$  and the connecting homomorphism
$$
\gamma_{i+1,j}:\mathrm{Tor}_{i+1}(I+L,K)_{(i+1)+j}\to \mathrm{Tor}_i(I\cap L,K)_{i+(j+1)}
$$
is surjective for all $i$ and all $j\geq d$.
\end{itemize}
If the equivalent conditions are satisfied, then $\beta_{i,i+j}(I)=\beta_{i,i+j}(I+L)$ for all $i$ and all $j>d$.
\end{lem}

\begin{proof}
(a)\implies (b): For all $j\geq d$, the short exact sequence
\begin{equation*}
0 \to I\cap L \to I \oplus L \to I+L\to 0
\end{equation*}
induces the long exact sequence
$$\begin{CD}
\cdots @>{\alpha_{i+1,j}}>> \Tor_{i+1}(I+ L,K)_{(i+1)+j} @>{\gamma_{i+1,j}}>> \Tor_i(I\cap L ,K)_{i+(j+1)} \\ 
@>>> \Tor_i(I,K)_{i+(j+1)} @>{\alpha_{i,j+1}}>> \Tor_{i}(I+ L,K)_{i+(j+1)} @>>> \cdots.
\end{CD}$$
Here we used the fact that  $\Tor_i(I \oplus L,K)_{i+(j+1)}=\Tor_i(I,K)_{i+(j+1)}$ for  all $j\geq d$. Condition (a) implies that $\Ker \left(\alpha_{i,j+1} \right)=0$ and hence $\gamma_{i+1,j}$ is surjective for all $j\geq d$.

The above sequence implies that $\Tor_i(I\cap L ,K)_{i+(j+1)}=0$ for all $i$ and all $j>d$, since the maps $\alpha_{i+1,j}$ and $\alpha_{i,j+1}$ are isomorphism for all $i$ and all $j>d$. Therefore $\reg (I \cap L) \leq d+1$.

\medskip
(b)\implies (a): For $j>d$ consider the exact sequence
$$\begin{CD}
\cdots @>{\gamma_{i+1,j-1}}>> \Tor_i(I \cap L ,K)_{i+j}@>>> \Tor_i(I,K)_{i+j} \\
 @>{\alpha_{i,j}}>> \Tor_{i}(I+ L,K)_{i+j} @>>> \Tor_{i-1}(I\cap L, K)_{(i-1)+(j+1)} @>>> \cdots.
\end{CD}$$

Our assumptions in (b) implies that $\Tor_{i-1}(I\cap L, K)_{(i-1)+(j+1)}=0$ and that $\gamma_{i+1,j-1}$ is surjective. Hence $\alpha_{i,j}$ is an isomorphism for all $i$ and all $j>d$.
\end{proof}

\begin{lem} \label{rem}
Let $J \subseteq I$ be homogeneous ideals generated in the same degree $d$. Then, 
\begin{itemize}
\item[\rm (a)] $\beta_{i, i+d}(J)\leq \beta_{i,i+d}(I)$, for all $i$.
\end{itemize}
Moreover, if for all $i$ and $j >d$, one has
\begin{equation} \label{extra assumption on betti numbers}
\beta_{i,i+j}(I) = \beta_{i,i+j}(J)
\end{equation}
then:
\begin{itemize}
\item[\rm (b)] $\reg \left( I \right) = \reg \left( J \right)$;  
\item[\rm (c)] $\Index \left( I \right) = \Index \left( J \right)$;
\item[\rm (d)] $\projdim (I) \leq \projdim (J)$.
\end{itemize}

\end{lem}

\begin{proof}
(a) From the short exact sequence
$$0 \longrightarrow J \longrightarrow I \longrightarrow \frac{I}{J} \longrightarrow 0$$
we get the long exact sequence:
\begin{align*}
\cdots \to \Tor_{i+1}(I/J,K)_{(i+1)+(d-1)} \to  \Tor_i(J,K)_{i+d} \stackrel{\phi} \longrightarrow \Tor_i(I,K)_{i+d} \to \cdots.
\end{align*}
Since $I/J$ is generated in degree $d$, $\Tor_{i+1}(I/J,K)_{(i+1)+(d-1)}=0$ and so $\phi$ is injective. Therefore, 
\begin{equation*}
\beta_{i,i+d}(J) = \dim_K \, \left( \Tor_i(J,K)_{i+d} \right) \leq \dim_K \, \left( \Tor_i(I,K)_{i+d} \right) = \beta_{i,i+d}(I).
\end{equation*}	

(b) Let $r:= \mathrm{reg} (J)$. Since $J$ is generated by elements of degree $d$, it follows that $r \geq d$. If $r=d$, then (\ref{extra assumption on betti numbers}) implies that $\beta_{i,j}(I) =0$, for all $j>i+d$. Hence, $\mathrm{reg} (I)=d = \mathrm{reg} (J)$. If $r>d$, then using (\ref{extra assumption on betti numbers}) again, we get the conclusion.

(c) If $\mathrm{index} (I) = \infty$, then $I$ has a $d$-linear resolution and by (b), the ideal $J$ has a $d$-linear resolution too. This is equivalent to say that $\mathrm{index}(J) = \infty$. In the case that $\mathrm{index} (I)$ is finite, (c) is again a direct consequence of (\ref{extra assumption on betti numbers}).

(d) Our assumption in (\ref{extra assumption on betti numbers}) together with part (a), implies that
\begin{equation*}
\beta_{i, i+j}(J) \leq \beta_{i, i+j}(I)
\end{equation*}
for all $i$ and $j$. This implies that,
\begin{align*}
\projdim (J) = \max\{i \colon  \beta_{i, i+j}(J) \neq 0 \text{ for some } j \} & \leq \max\{i \colon  \beta_{i, i+j}(I) \neq 0 \text{ for some } j \} \\
& =\projdim (I).
\end{align*}
\end{proof}

\begin{prop} \label{Cor1}
Let $I $ and $L$ be  graded ideals generated in degree $d$ such that both ideals $I \cap L$ and $L$ have a $d$-linear resolution. Then, 
\begin{itemize}
\item[\rm (a)] $\beta_{i,i+j} \left( I \right) = \beta_{i,i+j} \left( I+L \right)$, for  all $i$ and all $j>d$;
\item[\rm (b)] $\projdim \left( I+L \right) = \max \left\{ \projdim \left( I \right), \, \projdim  \left( L \right) \right\}$.
\end{itemize}
\end{prop}

\begin{proof}
Statement (a) is an immediate consequence of Lemma~\ref{ll}. To obtain (b), let $\rho = \projdim \left( I+L \right)$. Then $\beta_{\rho,\rho+j} \left( I+L \right)\neq 0$, for some $j \geq d$. If $j>d$, then $\beta_{\rho,\rho+j} \left( I \right) \neq 0$ by (a). Therefore in this case, $\max \left\{ \projdim (I), \, \projdim(L) \right\} \geq \rho$. 

Now assume that  $j=d$. The short exact sequence
\begin{equation*}
0 \longrightarrow I \cap L \longrightarrow I \oplus L \longrightarrow  I+L \longrightarrow 0
\end{equation*}
induces the long exact sequence
\begin{equation*}
\begin{array}{lll}
\cdots & \longrightarrow \Tor_{\rho+1}\left( I+L, \, K \right)_{\left( \rho+1 \right)+ \left( d-1 \right)} & \to \Tor_\rho \left( I \cap L, \, K \right)_{\rho+d}  \to \Tor_\rho \left( I \oplus L, \, K \right)_{\rho+d} \\
& \longrightarrow  \Tor_\rho \left( I+L, \, K \right)_{\rho+d} & \to \Tor_{\rho-1} \left(I \cap L, \, K \right)_{\left( \rho-1 \right) + \left( d+1 \right)}  \longrightarrow  \cdots.
\end{array}
\end{equation*}
Since the ideal $I \cap L$ has a $d$-linear resolution, $\Tor_{\rho-1} \left( I \cap L, \, K \right)_{\left( \rho - 1 \right) + \left( d+1 \right)} = 0$. Moreover, $\Tor_{\rho+1} \left( I+ L, \, K \right)_{\left( \rho+1 \right) + \left( d-1 \right)} = 0$. Thus, we obtain the following short exact sequence:
\begin{equation*}
0 \to  \Tor_\rho \left( I \cap L, \, K \right)_{\rho + d} \to \Tor_\rho \left( I, \, K \right)_{\rho+d} \oplus \Tor_\rho \left( L, \, K \right)_{\rho+d} \to \Tor_\rho \left( I+L, \, K \right)_{\rho+d} \to 0.
\end{equation*}
Since $\beta_{\rho, \rho+d} \left( I+L \right)\neq 0$ by assumption,  this short exact sequence implies that either $\beta_{\rho, \rho+d} \left( I \right) \neq 0$ or $\beta_{\rho, \rho+d} \left( L \right) \neq 0$. Hence, again in this case we have 
\begin{equation*}
\max \left\{ \projdim (I), \, \projdim (L) \right\} \geq \rho.
\end{equation*}
In order to prove the opposite inequality, we set $\rho' := \max \{ \projdim (I), \, \projdim(L) \}$. Suppose that $\projdim \left( L \right)= \rho'$. Then $\beta_{\rho', \rho'+j} \left( L \right) \neq 0$, for some $j \geq d$. Since $L$ has a $d$-linear resolution, $j=d$. Using Lemma~\ref{rem}(a), we have $\beta_{\rho', \rho'+d} \left( L \right) \leq \beta_{\rho', \rho'+d}\left( I+L \right)$. Thus $\projdim \left( I+L \right) \geq \rho'$.

Suppose now that  $\projdim \left( I \right)= \rho'$. Then  $\beta_{\rho', \rho'+j} \left( I \right) \neq 0$ for some $j \geq d$. If $j>d$, then $\beta_{\rho', \rho'+j} \left( I+L \right)\neq 0$ by (a). So $\projdim \left( I+L \right) \geq \rho'$. If $j=d$, then $\beta_{\rho', \rho'+d} \left( I \right) \leq \beta_{\rho', \rho'+d} \left( I+L \right)$, by Lemma~\ref{rem}(a). It follows that $\projdim \left( I+L \right) \geq \rho'$, in this case too.
\end{proof}

In the following proposition we discuss on the hypotheses of Proposition~\ref{Cor1}  in the case that $I$ and $L$ are  monomial ideals generated in degree $d$ and the generators of $L$ have a common factor of degree $d-1$. As we shall see, in this case, for the ideal $I \cap L$ the matter of having linear resolution is independent of the characteristic of the base field. To state this proposition, first we fix some notations.

For a given monomial $w=x_1^{c_1}\ldots x_n^{c_n}$, let $\nu_i(w)$ be the integer $c_i$. Also, for a monomial ideal $I$, let $G(I)$ denotes the unique minimal set of  monomial generators of $I$. Recall that, for monomial ideals $I$ and $L$, one has (see e.g. \cite[Proposition 1.2.1]{HHBook})
\begin{equation*} 
I \cap L = \left( {\rm lcm} \left( u, v \right) \colon \quad u \in G(I) \text{ and } v \in G(L) \right).
\end{equation*}

\begin{prop} \label{nice}
Let $0 \neq I \subset S= K[x_1, \ldots, x_n]$ be a monomial ideal generated in degree $d$, $u$ a non-zero monomial of degree $d-1$ and $\mathcal{L}$ a non-empty subset of  $\{x_1, \ldots, x_n\}$. Let $L= \left( x_iu \colon \quad x_i \in \mathcal{L} \right)$. The following conditions are equivalent:
\begin{itemize}
\item[\rm (a)] the ideal $I \cap L$ has a $d$-linear resolution;
\item[\rm (b)] $I \cap L$ is a non-zero monomial ideal generated in degree $d$;
\item[\rm (c)] either $L \subseteq I$, or for all $v \in G(I)$ there exists $x_i \in \mathcal{L}$ with $x_iu \in I$ such that $\nu_i \left( u \right) + 1 \leq \nu_i \left( v \right)$;
\item[\rm (d)] either $L \subseteq I$, or $I \cap L = \left( x_iu \colon \quad x_i \in \mathcal{L} \text{ and } x_iu \in I \right) \neq 0$.
\end{itemize}
\end{prop}

\begin{proof}
First of all note that:
\begin{equation} \label{Generators of intersection}
I \cap L = \left( {\rm lcm} \left( x_iu, v \right) \colon \quad x_i \in \mathcal{L} \text{ and } v \in G(I) \right).
\end{equation}

The implication (a)\implies (b) is obvious.

(b)\implies (c): Suppose that $L \nsubseteq I$ and take a generator $v \in G(I)$. The assumption $L \nsubseteq I$ implies that there exists $x_j\in \mathcal{L}$ such that  $x_ju\notin I$. In particular, $x_ju \neq v$ and $\deg \left( {\rm lcm} \left(x_ju, v \right) \right) >d$. Since $I \cap L \neq 0$ is generated in degree $d$, it follows that there exists $x_i \in \mathcal{L}$ such that $x_iu \in G(I)$ and $x_iu$ divides ${\rm lcm}(v, x_ju)$. Note that $i \neq j$, for $x_iu \in I$. Furthermore,
\begin{equation*}
1 + \nu_i \left( u \right) = \nu_i\left( x_iu \right) \leq \max \left\{ \nu_i\left( x_ju \right), \nu_i\left( v \right)  \right\} = \max \left\{ \nu_i \left( u \right), \nu_i\left( v \right) \right\}.
\end{equation*}
The above inequality shows that $\nu_i \left( u \right) + 1 \leq \nu_i \left( v \right)$.

(c)\implies (d): If $L \subseteq I$, there exists nothing to prove. So assume that $L \nsubseteq I$ and let 
\begin{equation*}
L' = \left( x_iu \colon \quad x_i \in \mathcal{L} \text{ and } x_iu \in I \right).
\end{equation*}
By~(\ref{Generators of intersection}), it is enough to show that ${\rm lcm} \left( x_ju, v \right) \in L'$, for all $x_j \in \mathcal{L}$ and $v \in G(I)$.

Let $v \in G(I)$ and $x_j \in \mathcal{L}$. Our assumption implies that there exists $x_i \in \mathcal{L}$ with $x_iu \in I$ such that $\nu_i \left( u \right) + 1 \leq \nu_i \left( v \right)$. So, $x_iu \in L'$ and
\begin{itemize}
\item[] $\nu_i \left( x_iu \right) = \nu_i \left( u \right) + 1 \leq \nu_i \left( v \right)$;
\item[] $\nu_k \left( x_iu \right) = \nu_k \left( u \right) \leq \nu_k \left( x_ju \right)$, for $k \neq i$.
\end{itemize}
Thus $\nu_k \left( x_iu \right) \leq \max \left\{ \nu_k \left( x_ju \right), \, \nu_k \left( v \right) \right\} = \nu_k \left( {\rm lcm} \left( x_ju, v \right) \right)$, for all $k=1, \ldots, n$. This implies that, $x_iu$ divides ${\rm lcm} \left( x_ju, v \right)$.

(d)\implies (a): The assumption implies that:
\begin{equation*}
I \cap L = \left( x_iu \colon  \quad x_i \in \mathcal{L}' \right),
\end{equation*}
where $\mathcal{L}'$ is a non-empty subset of $\{x_1, \ldots, x_n\}$. One may easily verify that such ideals have a linear resolutions.
\end{proof}

\begin{rem} \label{rem on nice}
Let $0 \neq I \subset S= K[x_1, \ldots, x_n]$ be a square-free monomial ideal generated in degree $d$, $u$ a non-zero square-free monomial of degree $d-1$ and $\mathcal{L}$ a non-empty subset of  $\{x_1, \ldots, x_n\}$ such that,
\begin{itemize}
\item $x_i$ does not divide $u$ for all $x_i \in \mathcal{L}$,
\item $L = \left( x_iu \colon \quad x_i \in \mathcal{L} \right) \nsubseteq I$.
\end{itemize}
Then, $L$ has a $d$-linear resolution with $\projdim \left( L \right) = |\mathcal{L}| - 1$ and Proposition~\ref{nice} implies the equivalence of the following statements:
\begin{itemize}
\item[\rm (a)] the ideal $I \cap L$ has a $d$-linear resolution;
\item[\rm (b)] for all $v \in G(I)$ there exists $x_i \in \mathcal{L}$ such that $x_iu \in I$ and $x_i$ divides $v$.
\end{itemize}
\end{rem}

The following corollary is an essential part of proving the main theorem (Theorem~\ref{main}) of this section.
\begin{cor} \label{Essential of main theorem}
Let $\mathcal{C}$ be a $d$-uniform clutter on the vertex set $[n]$,  $e$ a simplicial element of $\C$ and $\C'=\C \setminus e$. Let $I=I \left( \bar{\C} \right)$ and $J=I \left( \bar{\C}' \right)$ be the corresponding circuit ideals. Then,
 \begin{itemize}
 	\item[\rm (a)] $\beta_{i,i+j} \left( I  \right) =\beta_{i,i+j} \left( J \right)$, for all $i$ and all $j>d$.
 	\item[\rm (b)] $\projdim (J) \geq \projdim (I)$. Indeed, if $e \in {\rm SC}(\C)$, then $\projdim\left( J \right) = n-d$.
 \end{itemize}
\end{cor}
 
\begin{proof}
If $e \notin {\rm SC}(\C)$, then $\C \setminus e = \C$ and there is nothing to prove. So assume that $e \in {\rm SC}(\C)$ is simplicial. In this case, let $u=\prod_{i \in e}x_i$ and $L=\left( x_iu \colon \quad i \in [n] \setminus e \right)$. Then $L$ has a $d$-linear resolution with $\projdim \left( L \right) = n-d$ and $J=I+L$. We claim that $I \cap L$ has a $d$-linear resolution.
To prove our claim, first note that $L \nsubseteq I$, because $e \in {\rm SC}(\C)$. So, it is enough to show that statement (b) in Remark~\ref{rem on nice} is satisfied for the ideal $I$ and the subset $\mathcal{L} = \left\{x_i \colon i \in [n] \setminus e \right\}$ of variables. Assume that $v \in G(I)$. Then $v={\textbf x}_F$ for some $F\in \bar{\mathcal{C}}$. Since $e$ is a simplicial  submaximal circuit of $\mathcal{C}$, we have $F \nsubseteq {N}_\C[e]$. Take an element $i \in F \setminus N_\C[e]$. Then $x_i \in \mathcal{L}$, $\left\{ i \right\} \cup e \in \bar{\C}$, $x_i|v$ and $x_iu \in I$. This completes the proof of the claim.

Since $L$ and $I \cap L$ have  linear resolutions, statements (a) and (b) can be obtained from Proposition~\ref{Cor1}.
\end{proof}

Now we are ready to prove our main theorem.

\medskip
\begin{proof}[Proof of  Theorem~$\ref{main}$] 
	If $A = \{ F \in \C \colon \quad e \subset F \}$, then the Corollary~\ref{Essential of main theorem}(a), yields the conclusion. So assume that $A':= \{ F \in \C \colon \quad e \subset F \text{ and } F \notin A \} \neq \varnothing$. Then, it is clear that $\mathcal{D} = \left( \C \setminus e \right) \cup A'$ and $e$ is a simplicial submaximal circuit of $\mathcal{D}$ (and $\C$). Moreover, $\mathcal{D} \setminus e = \C \setminus e$. Hence, by Corollary~\ref{Essential of main theorem}(a), we conclude that:
\begin{equation*}
\beta_{i,i+j} \left( I \left( \bar{\mathcal{D}} \right) \right) = \beta_{i,i+j} \left( I \left( \overline{\mathcal{D} \setminus e} \right) \right) = \beta_{i,i+j} \left( I \left( \overline{\C \setminus e} \right) \right) = \beta_{i,i+j} \left( I \left( \bar{\C} \right) \right)
\end{equation*}	
for all $i$ and $j>d$, as desired. The other assertions are direct consequences of Lemma~\ref{rem}.
\end{proof}

\section{chordal clutters} \label{Section Chordal}
Chordal graphs are probably the most important combinatorial objects in the problem of  classification of monomial ideals with linear resolution (quotients). Thanks to Fr\"oberg, we know that a square-free monomial ideal $I$ generated in degree $2$ has a linear resolution (over any field $K$) if and only if   $I=I\left( \bar{G} \right)$, for some chordal graph $G$ (\cite{Fr}). Later, this result has been improved by showing that $I$ has linear quotients if and only if   $I=I\left( \bar{G} \right)$, where $G$ is a chordal graph (\cite{HHZ}).

This successful combinatorial characterization of ideals generated in degree $2$ with linear resolution (quotients), motivated many mathematicians to generalize this result for square-free monomial ideals generated in degree $d>2$. Some partial results on this subject are given in \cite{Adiprasito, Nevo, ConnonFaridi, Emtander, HaVanTuyl, MNYZ, MYZ, VanTuyl-Villarreal, Woodroofe}. The goal of this section is to introduce a class of $d$-uniform clutters (which we call chordal clutters) which extends the definition of the class of chordal graphs, and the corresponding circuit ideal has a linear resolution over any field. The idea to state the definition in this manner comes from Theorem~\ref{main}(a).

\begin{defn}\rm
Let $\mathcal{C}$ be a $d$-uniform clutter. We call $\mathcal{C}$ a \textit{chordal clutter}, if either $\mathcal{C}=\varnothing$, or $\mathcal{C}$ admits a simplicial submaximal circuit $e$ such that $\mathcal{C}\setminus e$ is chordal.
\end{defn}
Following the notation in \cite{MNYZ}, we use $\mathfrak{C}_d$, to denote the class of all $d$-uniform chordal clutters. Our aim in this section is to show that, this definition is the natural generalization of chordal graphs. That is, with this definition, we show that the circuit ideal associated to these clutters has a linear resolution over any field. Note that the definition of chordal clutters can be restated as follows:

The $d$-uniform clutter $\mathcal{C}$ is chordal if either $\mathcal{C} = \varnothing$, or else there exists a sequence of submaximal circuits of $\mathcal{C}$, say $e_1, \ldots, e_t$, such that $e_1$ is simplicial submaximal circuit of $ \mathcal{C}$, $e_i$ is simplicial submaximal circuit of $\left( \left( \left( \mathcal{C} \setminus e_{1} \right) \setminus e_2 \right)  \setminus \cdots \right) \setminus e_{i-1}$ for all $i>1$, and $\left( \left( \left(  \mathcal{C} \setminus e_{1} \right) \setminus e_2 \right) \setminus \cdots \right) \setminus e_{t} = \varnothing$. 

To simplify the notation, we use $\mathcal{C}_{e_1 \dots e_i}$ for $\left( \left( \left( \mathcal{C} \setminus e_{1} \right) \setminus e_2 \right) \setminus \cdots \right) \setminus e_{i}$.

\begin{ex}
In Figure~\ref{F2}, the $3$-uniform clutter $\mathcal{C}$ is chordal, while  the $3$-uniform clutter $\mathcal{D}$ is not.
	\begin{align*}
	&\mathcal{C}=\{ \{ 1,2,3 \}, \{ 1,2,4 \}, \{ 1,3,4 \}, \{ 2,3,4 \}, \{ 1,2,5 \},
	\{ 1,2,6 \}, \{ 1,5,6 \}, \{ 2,5,6 \} \}. \\
	&\mathcal{D}=\{\{1,2,3\}, \{1,2,4\}, \{1,3,4\},\{2,3,5\}, \{2,4,5\}, \{3,4,5\}\}.
	\end{align*}
	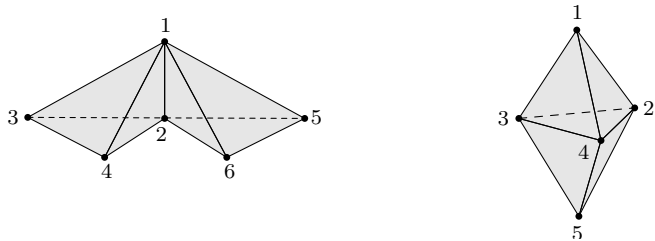
\begin{figure}[ht!]
		\begin{tikzpicture}[line cap=round,line join=round,>=triangle 45,x=0.7cm,y=0.7cm]
		\clip(5.26,2.19) rectangle (12.44,6.8);
		\fill[fill=black,fill opacity=0.1] (8.8,5.94) -- (6.2,4.5) -- (7.66,3.74) -- cycle;
		\fill[fill=black,fill opacity=0.1] (8.8,5.94) -- (8.8,4.48) -- (7.66,3.74) -- cycle;
		\fill[fill=black,fill opacity=0.1] (8.8,5.94) -- (11.46,4.48) -- (9.98,3.74) -- cycle;
		\fill[fill=black,fill opacity=0.1] (8.8,5.94) -- (8.8,4.48) -- (9.98,3.74) -- cycle;
		\draw (8.8,5.94)-- (6.2,4.5);
		\draw (6.2,4.5)-- (7.66,3.74);
		\draw (7.66,3.74)-- (8.8,5.94);
		\draw (8.8,5.94)-- (8.8,4.48);
		\draw (8.8,4.48)-- (7.66,3.74);
		\draw (7.66,3.74)-- (8.8,5.94);
		\draw (8.8,5.94)-- (11.46,4.48);
		\draw (11.46,4.48)-- (9.98,3.74);
		\draw (9.98,3.74)-- (8.8,5.94);
		\draw (8.8,5.94)-- (8.8,4.48);
		\draw (8.8,4.48)-- (9.98,3.74);
		\draw (9.98,3.74)-- (8.8,5.94);
		\draw [dash pattern=on 2pt off 2pt] (6.2,4.5)-- (11.46,4.48);
		\draw (8.52,6.58) node[anchor=north west] {\begin{scriptsize}1\end{scriptsize}};
		\draw (8.44,4.50) node[anchor=north west] {\begin{scriptsize}2\end{scriptsize}};
		\draw (5.62,4.83) node[anchor=north west] {\begin{scriptsize}3\end{scriptsize}};
		\draw (7.4,3.8) node[anchor=north west] {\begin{scriptsize}4\end{scriptsize}};
		\draw (11.38,4.82) node[anchor=north west] {\begin{scriptsize}5\end{scriptsize}};
		\draw (9.74,3.8) node[anchor=north west] {\begin{scriptsize}6\end{scriptsize}};
		\begin{scriptsize}
		\fill [color=black] (8.8,5.94) circle (1.3pt);
		\fill [color=black] (6.2,4.5) circle (1.3pt);
		\fill [color=black] (7.66,3.74) circle (1.3pt);
		\fill [color=black] (8.8,4.48) circle (1.3pt);
		\fill [color=black] (11.46,4.48) circle (1.3pt);
		\fill [color=black] (9.98,3.74) circle (1.3pt);
		\end{scriptsize}
		\end{tikzpicture}
		\quad
		\quad
		\quad
		\begin{tikzpicture}[line cap=round,line join=round,>=triangle 45,x=.7cm,y=.7cm]
		\clip(13.89,2.19) rectangle (18.76,6.8);
		\fill[fill=black,fill opacity=0.1] (16.24,6.16) -- (15.14,4.48) -- (16.7,4.06) -- cycle;
		\fill[fill=black,fill opacity=0.1] (16.24,6.16) -- (17.34,4.68) -- (16.7,4.06) -- cycle;
		\fill[fill=black,fill opacity=0.1] (15.14,4.48) -- (16.28,2.62) -- (16.7,4.06) -- cycle;
		\fill[fill=black,fill opacity=0.1] (17.34,4.68) -- (16.28,2.62) -- (16.7,4.06) -- cycle;
		\draw (16.24,6.16)-- (15.14,4.48);
		\draw (15.14,4.48)-- (16.7,4.06);
		\draw (16.7,4.06)-- (16.24,6.16);
		\draw (16.24,6.16)-- (17.34,4.68);
		\draw (17.34,4.68)-- (16.7,4.06);
		\draw (16.7,4.06)-- (16.24,6.16);
		\draw [dash pattern=on 3pt off 3pt] (15.14,4.48)-- (17.34,4.68);
		\draw (15.14,4.48)-- (16.28,2.62);
		\draw (16.28,2.62)-- (16.7,4.06);
		\draw (16.7,4.06)-- (15.14,4.48);
		\draw (17.34,4.68)-- (16.28,2.62);
		\draw (16.28,2.62)-- (16.7,4.06);
		\draw (16.7,4.06)-- (17.34,4.68);
		\draw (15.95,6.83) node[anchor=north west] {\begin{scriptsize}1\end{scriptsize}};
		\draw (17.31,5) node[anchor=north west] {\begin{scriptsize}2\end{scriptsize}};
		\draw (14.55,4.8) node[anchor=north west] {\begin{scriptsize}3\end{scriptsize}};
		\draw (16.07,4.17) node[anchor=north west] {\begin{scriptsize}4\end{scriptsize}};
		\draw (15.95,2.62) node[anchor=north west] {\begin{scriptsize}5\end{scriptsize}};
		\begin{scriptsize}
		\fill [color=black] (16.24,6.16) circle (1.3pt);
		\fill [color=black] (15.14,4.48) circle (1.3pt);
		\fill [color=black] (16.7,4.06) circle (1.3pt);
		\fill [color=black] (17.34,4.68) circle (1.3pt);
		\fill [color=black] (16.28,2.62) circle (1.3pt);
		\end{scriptsize}
		\end{tikzpicture}
		\caption{The Clutter $\C$ on the left and $\mathcal{D}$ on the right}
		\label{F2}
	\end{figure}
\end{ex}

\begin{rem}
It is worth to say that this definition of chordal clutters coincides with the graph theoretical definition of chordal graphs in the case $d=2$. To see this, we recall that, a graph $G$ is chordal, if and only if for every  induced subgraph $G'$ of $G$, one has $\mathrm{Simp} \left(G' \right) \neq \varnothing$ (essentially \cite{Dirac}). It follows that, a graph $G$ is chordal, if and only if   there exists an order on vertices of $G$, say $v_1, \ldots v_n$, such that $v_i$ is simplicial in $G|_{\left\{ v_1, \ldots, v_{i-1} \right\}}$. This is equivalent to say that $G \in \mathfrak{C}_2$.
\end{rem}

The following result, which is contained in \cite[Remark 3.10]{MNYZ}, justifies our definition of chordality, because the property of being  chordal should, as in Fr\"oberg's theorem, imply that the circuit ideal has a $d$-linear resolution over any field $K$.

\begin{thm}	\label{chordallin}
Let $\mathcal{C} \neq \mathcal{C}_{n,d}$ be a $d$-uniform chordal clutter on the vertex set $[n]$. Then $I(\bar{\mathcal{C}})$ has a $d$-linear resolution over any field $K$.
\end{thm}

\begin{proof}
Since $\mathcal{C}$ is chordal, there exists a sequence of submaximal circuits of $\mathcal{C}$, say $e_1, \ldots, e_t$, such that $e_1$ is simplicial submaximal circuit of $\mathcal{C}$, $e_i$ is simplicial submaximal circuit of $\mathcal{C}_{e_1\dots e_{i-1}}$ for all $i>1$, and $\mathcal{C}_{e_1\dots e_t}=\varnothing$. It follows from Theorem~\ref{main}(b) that $I \left( \bar{\mathcal{C}} \right)$ has a $d$-linear resolution if and only if $I \left( \overline{\mathcal{C}_{e_1\dots e_t}} \right)$ has a $d$-linear resolution. This is indeed the case, because $I \left( \overline{\mathcal{C}_{e_1\dots e_t}} \right) = I \left( \bar{\varnothing} \right) = I \left( \mathcal{C}_{n,d} \right)$.
\end{proof}

As mentioned at the beginning of this section, the nice characterization of square-free monomial ideals generated in degree $2$ in terms of chordal graphs (Fr\"oberg's theorem), motivated many mathematicians to generalize the definition of chordal graphs. Some of the most important results are due to Emtander \cite{Emtander}, Woodroofe \cite{Woodroofe}, Connon and Faridi \cite{ConnonFaridi}, and Nevo et al \cite{Nevo} (ordered chronologically). In the remaining of this section, we compare the class $\mathfrak{C}_d$ with the other known families of chordal clutters.

\subsection{Woodroofe's chordal class vs $\mathfrak{C}_d$}
A nice class of clutters whose circuit ideals have a linear resolution over any field  (in fact have linear quotients) has been defined by Woodroofe in \cite{Woodroofe}. This class is named chordal clutters in that text and for the avoidance of ambiguity, we call it W-chordal in this paper. Below, we state the definition of a W-chordal clutter and show that this class is strictly contained in $\mathfrak{C}_d$. To state the definition of W-chordal clutters, first we need the following operations as defined in \cite{Woodroofe}.

Given a clutter $\C$ (not necessarily uniform), there are two ways of removing a vertex that are of interest. Let $v \in V(\C)$. The \textit{deletion}, $\C \setminus v$, is the clutter on the vertex set $V(\C) \setminus \{v\}$, with circuits
$\{F \in \C \colon \quad v \notin F\}$. The \textit{contraction}, $\C/v$, is the clutter on the vertex set $V(\C)\setminus \{v\}$, which the circuits are the minimal sets of $\{ F \setminus \{v\} \colon \quad F \in \C \}$. Thus, $\C \setminus v$ deletes all circuits containing $v$, while $\C/v$ removes $v$ from each circuit containing it and then removes any redundant circuits.

A clutter $\mathcal{D}$ obtained from $\C$ by a sequence of deletions and/or contractions is called a \textit{minor}
of $\C$. It is straightforward to prove that, if $v \neq w$ are vertices, then:
$$(\C \setminus v) \setminus w = (\C \setminus w) \setminus v, \quad (\C/v)/w = (\C/w)/v, \quad (\C \setminus v)/ w = (\C/w) \setminus v.$$

\begin{defn}[W-chordal] \label{Definition of W-Chordal}
Let $\C$ be a clutter. A vertex $v$ of $\C$ is \textit{W-simplicial}, if for every two circuits $F_1$ and $F_2$ of $\C$ that contain $v$, there is a third circuit $F_3$ such that, $F_3 \subseteq (F_1 \cup F_2) \setminus \{v\}$. A clutter $\C$ is called W-\textit{chordal}, if every minor of $\C$ has a W-simplicial vertex.
\end{defn}

As it is mentioned in \cite{Woodroofe}, W-chordal clutters contain a variety of classes of combinatorial objects, including chordal graphs, complete $d$-uniform clutters, matroids, etc.  To show that W-chordal clutters are contained in $\mathfrak{C}_d$, we need the following intermediate steps.

\begin{lem} \label{W-chordals are chordal, first lemma}
Let $\C$ be a $d$-uniform clutter and $e$ be a submaximal circuit of $\C$. Then, $e$ is
simplicial if and only if each vertex of $e$ is W-simplicial in $\C|_{N_\C \left[ e \right]}$.
\end{lem}

\begin{proof}
The implication ($\Rightarrow$) is straightforward. For the other direction, assume that each vertex of $e$ is W-simplicial in $\C|_{N_\C \left[ e \right]}$. We must show that each $d$-subset of $N_\C[e]$ is a circuit in $\C$. Suppose that $e = \left\{ v_1, \ldots, v_{d-1} \right\}$ and let $G$ be a $d$-subset of $N_\C[e]$. The proof is by induction on $i = |G \setminus e|$. For $i = 1$, we have then $e \subset G$, and $G$ is a circuit by definition of $N_\C[e]$. Suppose that $i > 1$. Without loss of generality, we may assume that $G = \left\{ v_i, v_{i+1}, \ldots, v_{d-1}, w_1, \ldots ,w_i \right\}$. Then, the sets $G_1 = \left\{ v_1 \right\} \cup \left( G \setminus \left\{ w_1 \right\} \right)$ and $G_2 = \left\{ v_1 \right\} \cup \left( G \setminus \left\{ w_2 \right\} \right)$ are circuits in $\C$ by the induction hypothesis. On the other hand, $v_1 \in G_1 \cap G_2$ and $v_1$ is W-simplicial and therefore $G = \left( G_1 \cup G_2 \right) \setminus \left\{ v_1 \right\}$ is a circuit in $\C$.
\end{proof}

\begin{lem} \label{W-chordals are chordal, second lemma}
Let $\C$ be a $d$-uniform clutter, and $v_1$ be a W-simplicial vertex such that $\C|_{N_\C \left[ \left\{ v_1 \right\} \right]}$ is W-chordal. Then, there exists a simplicial submaximal circuit $e$ such that $v_1 \in e$.
\end{lem}

\begin{proof}
The clutter $\left( \C|_{N_\C \left[ \left\{ v_1 \right\} \right]} \right) / v_1$ is a minor of $\C|_{N_\C \left[ \left\{ v_1 \right\} \right]}$ and has a W-simplicial vertex, namely $v_2$. In this case, $v_2$ is also a W-simplicial vertex in $\C_1 := \C|_{N_\C \left[ \left\{ v_1, v_2 \right\} \right]} / v_1$. Moreover, the clutter $\C_1/ v_2$ is again a minor of $\C|_{N_\C \left[ \left\{ v_1 \right\} \right]}$ and has a W-simplicial vertex as $v_3$ which is a W-simplicial vertex in $\C_2 := \C|_{N_\C \left[ \left\{ v_1, v_2, v_3 \right\} \right]} /v_1 / v_2$ too. By this process, we obtain a set of vertices $e := \left\{ v_1, v_2, \ldots, v_{d-1} \right\}$ which is a submaximal circuit of $\C$. We prove that for each $1 \leq i \leq d - 1$, the vertex $v_i$ is W-simplicial in $\C|_{N_\C \left[ e \right]}$. Then, by Lemma~\ref{W-chordals are chordal, first lemma}, we may deduce that $e$ is simplicial. To do this, first we claim that:
\begin{itemize}
\item[]\textbf{Claim. }
For any $(d-1)$-subset $B = \left\{ w_1, \ldots, w_{d-1} \right\}$ of $N_\C \left[ e \right] \setminus \left\{ v_1 \right\}$, the set $B \cup \left\{ v_1 \right\}$ is a circuit in $\C|_{N_\C \left[ e \right]}$.
\end{itemize}
\begin{itemize}
\item[] \textit{Proof of the claim.} Without loss of generality we may assume that $B \cap e = \{ v_1, \ldots, v_{t-1} \}$ and $B \setminus e = \left\{ w_t, w_{t+1}, \ldots, w_{d-1} \right\}$, for some integer $1 \leq t \leq d-1$. If $|B \setminus e | = 1$, then $\left\{ v_1 \right\} \cup B = e \cup \left\{ w_{d-1} \right\}$ which is obviously a circuit in $\C|_{N_\C \left[ e \right]}$.
Let $|B \setminus e | > 1$. Then, $e \cup \left\{ w_{d-1} \right\}$ and $e \cup \left\{ w_{d-2} \right\}$ are circuits in $\C|_{N_\C \left[ e \right]}$. Therefore, $\left\{ v_{d-1}, w_{d-1} \right\}$ and $\left\{ v_{d-1}, w_{d-2} \right\}$ are circuits in the minor $\C|_{N_\C \left[ e \right]} / v_1 / v_2 / \cdots / v_{d-2}$. In this minor, $v_{d-1}$ is also W-simplicial and hence $\left\{ w_{d-1}, w_{d-2} \right\}$ is a circuit in this minor. Therefore, $\left\{ v_1, \ldots, v_{d-2}, w_{d-1}, w_{d-2} \right\}$ is a circuit in $\C|_{N_\C \left[ e \right]}$. Similarly $\left\{ v_1, \ldots, v_{d-2}, w_{d-1}, w_{d-3} \right\}$ is a circuit in $\C|_{N_\C \left[ e \right]}$. Thus, $\left\{ v_{d-2}, w_{d-1}, w_{d-3} \right\}$ and $\left\{ v_{d-2}, w_{d-1}, w_{d-2} \right\}$ are circuits in the minor $\C|_{N_\C \left[ e \right]} / v_1 / \cdots / v_{d-3}$ which has $v_{d-2}$ as a W-simplicial vertex. Now, as the above argument, we conclude that $\left\{ v_1, \ldots, v_{d-3}, w_{d-1}, w_{d-2}, w_{d-3} \right\}$ is a circuit in $\C|_{N_\C \left[ e \right]}$. Continuing this process, we obtain the conclusion.
\end{itemize}

Now, we are ready to show that, for each $1 \leq i \leq d - 1$, the vertex $v_i$ is W-simplicial in $\C|_{N_\C \left[ e \right]}$. Note that $v_1$ is a W-simplicial vertex of $\C$ and hence is a W-simplicial in $\C|_{N_\C \left[ e \right]}$. For $1 < i \leq d-1$, let $E$ and $F$ be two circuits in $\C|_{N_\C \left[ e \right]}$ containing $v_i$. 
\begin{itemize}

\item If $v_1 \notin E \cup F$, then $\left\{ v_1 \right\} \cup \left( E \setminus \left\{ v_i \right\} \right)$ and $\left\{ v_1 \right\} \cup \left( F \setminus \left\{ v_i \right\} \right)$ are circuits containing the W-simplicial vertex $v_1$ and therefore, there is a circuit contained in $\left( E \cup F \right) \setminus \left\{ v_i \right\}$.
\item If $v_1 \in E \cap F$, then for each ($d-1$)-subset $D$ of $\left( E \cup F \right) \setminus \left\{ v_1, v_i \right\}$, we have $D \cup \left\{ v_1 \right\}$ is a circuit in $\C|_{N \left[ e \right]}$ which is also contained in $\left( E \cup F \right) \setminus \left\{ v_i \right\}$.
\item If $v_1 \in E$ and $v_1 \notin F$, then there exists $u \in F \setminus E$ such that $\left( E \setminus \left\{ v_i \right\} \right) \cup \left\{ u \right\}$ and $\left( F \setminus \left\{ v_i \right\} \right) \cup \left\{ v_1 \right\}$ are circuits in $\C$ which contain $v_1$. Hence, there is a circuit inside $\left( E \cup F \right) \setminus \left\{ v_i \right\}$.
\end{itemize}
\end{proof}

The reader may notice that, in the middle of the proof of Lemma~\ref{W-chordals are chordal, second lemma}, it is actually proved the following statement:
\begin{cor} \label{W-chordal are chordal-third lemma}
Let $\C$ be a $d$-uniform clutter and $v_1, \ldots, v_{d-1}$ be a sequence of vertices of $\C$  such that:
\begin{itemize}
\item[\rm (i)] $v_1$ is W-simplicial in $\C$;
\item[\rm (ii)] $v_i$ is W- simplicial in $\left( \C|_{N_\C \left[ \{v_1, \ldots, v_{i-1} \} \right]} \right) / v_1 / \cdots/ v_{i-1}$, for $i=2, \ldots, d-1$;
\end{itemize}
Then $e= \{v_1, \ldots, v_{d-1} \}$ is a simplicial submaximal circuit in $\C$.
\end{cor}

\begin{prop}[\textbf{W-chordals are chordal}] \label{W-chordals are chordal}
If $\C$ is a $d$-uniform W-chordal clutter, then $\C \in \mathfrak{C}_d$.
\end{prop}

\begin{proof}
Since $\C$ is W-chordal, $\C$ has a W-simplicial vertex $v$. Let $e_1$ be a simplicial submaximal circuit as in the proof of Lemma~\ref{W-chordals are chordal, second lemma}, which contains $v$. It is easy to see that, $v$ is W-simplicial in $\C \setminus e_1$, if $v$ is still a non-isolated vertex in $\C \setminus e_1$. Let $e_1 = \left\{ v, w_1, w_2, \ldots, w_{d-2} \right\}$ and $e'_1 = \left\{ v, w_1, \ldots, w_{d-3} \right\}$. Since $e'_1 \cup \{ w_{d-2} \} = e_1$, there is no circuit in $\C\setminus e_1$ containing $e'_1 \cup \{ w_{d-2} \}$. Hence,
\begin{align*}
 \left( \left( \C \setminus e_1 \right)|_{N_{\C \setminus e_1} \left[ e'_1 \right]} \right) / v / w_1 / \cdots / w_{d-3}  &= 
\left( \left( \left( \C \setminus e_1 \right)|_{N_{\C \setminus e_1} \left[ e'_1 \right]} \right) / v / w_1 / \cdots / w_{d-3} \right) \setminus w_{d-2} \\
& = \left( \left( \C|_{N_{\C \setminus e_1} \left[ e'_1 \right]} \right) / v / w_1 / \cdots / w_{d-3} \right) \setminus w_{d-2}
\end{align*}
is a minor of $\C$. Pick a W-simplicial vertex $w'_{d-2}$ in this minor. We claim that $e_2 = \left\{ v, w_1, w_2, \ldots, w_{d-3}, w'_{d-2} \right\}$ is a simplicial submaximal circuit in $\C \setminus e_1$. 
\begin{itemize}
\item[] \textit{Proof of the claim}. We show that the vertices of $e_2$ satisfy the conditions of Corollary~\ref{W-chordal are chordal-third lemma}. Note that, by our choice of $v$ and $w'_{d-2}$, these vertices satisfy in Corollary~\ref{W-chordal are chordal-third lemma}. Now, let:
\begin{align*}
& \mathcal{D}_1 := \left( \left( \C \setminus e_1 \right)|_{N_{\C \setminus {e_1}} \left[ \{ v \}\right]} \right)/ v, 
& \mathcal{D}'_1 := \left( \C |_{N_{\C} \left[ \{ v \} \right]} \right)/ v, 
\end{align*}
and for $i>1$, let 
\begin{align*}
& \mathcal{D}_i := \left( \left( \C \setminus {e_1} \right)|_{N_{\C \setminus {e_1}} \left[ \{ v, w_1, \ldots, w_{i-1} \}\right]} \right) / v / w_1/ \cdots/ w_{i-1}, \\
& \mathcal{D}'_i := \left( \C |_{N_{\C} \left[ \{ v, w_1, \ldots, w_{i-1} \} \right]} \right) / v / w_1/ \cdots/ w_{i-1}.
\end{align*}
It is enough to show that $w_i$ is a W-simplicial vertex in $\mathcal{D}_i$, for $i=1, \ldots, d-3$. Let $F_1$ and $F_2$ be circuits in $\mathcal{D}_i$ which contain $w_i$. Since $w_i$ is W-simplicial in $\mathcal{D}'_i$, there exists a circuit $F_3 \in \mathcal{D}'_i$ such that $F_3 \subset \left( F_1 \cup F_2 \right) \setminus \{w_i\}$. Then, $F_3 \subset N_{\C \setminus {e_1}} \left[ \{ v, w_1, \ldots, w_{i-1} \}\right]$, $w_i \notin F_3$ and $F_3 \cup A \in \C |_{N_{{\C} \setminus {e_1}} \left[ \{ v, w_1, \ldots, w_{i-1} \} \right]}$, for a subset $A \subseteq \{ v, w_1, \ldots, w_{i-1} \}$. Since $w_i \notin F_3$, we conclude that ${e_1} \nsubseteq F_3 \cup A$. Hence 
\[
F_3 \cup A \in \left( \C \setminus {e_1} \right)|_{N_{\C \setminus {e_1}} \left[ \{ v, w_1, \ldots, w_{i-1} \} \right]}.
\]
This means that $F_3 \in \mathcal{D}_i$ is a circuit which is contained in $\left( F_1 \cup F_2 \right) \setminus \{w_i\}$, as desired.
\end{itemize}

Continuing this process, by removing all simplicial submaximal circuits containing $e'_1$, after a finite number of steps, say $r$, the clutter $\C_{e_1 \ldots e_r}$ has no circuit containing $e'_1$. Now we put $e'_2 = e'_1 \setminus \{w_{d-3}\}$ and we do the same as above, to obtain a subclutter of $\C_{e_1 \ldots e_r}$ which has no circuit containing $e'_2$. Repeating this argument,  finally, the vertex $v$ will be an isolated vertex and the remaining clutter is $\C \setminus v$ which is a $d$-uniform W-chordal clutter. Now induction completes the proof.
\end{proof}

The following example shows that the containment in Proposition~\ref{W-chordals are chordal} is strict.

\begin{ex}
Let $\C$ be the following $3$-uniform clutter with vertex set $\left\{ 1, \ldots, 5 \right\}$:
\[
\C=\left\{ 123, 134, 235, 345 \right\},
\]
observing $\left\{1, 2, 3 \right\}$ by $123$ and so on.

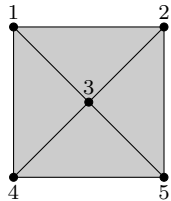
\begin{figure}[H]
\centering
\begin{tikzpicture}[line cap=round,line join=round,>=triangle 45,x=1.0cm,y=1.0cm]
\fill[color=gray,fill=gray,fill opacity=0.4] (1.99,0.99) -- (3.99,0.99) -- (3.99,2.99) -- (1.99,2.99) -- cycle;
\draw (2.,1.)-- (4.,1.);
\draw (4.,1.)-- (4.,3.);
\draw (4.,3.)-- (2.,3.);
\draw (2.,3.)-- (2.,1.);
\draw (2.,3.)-- (4.,1.);
\draw (4.,3.)-- (2.,1.);
\begin{scriptsize}
\draw [fill=black] (2.,3.) circle (1.5pt);
\draw[color=black] (2.0,3.2) node {$1$};
\draw [fill=black] (4.,3.) circle (1.5pt);
\draw[color=black] (4.0,3.2) node {$2$};
\draw [fill=black] (3.,2.) circle (1.5pt);
\draw[color=black] (3.0,2.2) node {$3$};
\draw [fill=black] (2.,1.) circle (1.5pt);
\draw[color=black] (2.0,0.8) node {$4$};
\draw [fill=black] (4.,1.) circle (1.5pt);
\draw[color=black] (4.0,0.8) node {$5$};
\end{scriptsize}
\end{tikzpicture}
\caption{A clutter $\C$ in $\mathfrak{C}_3$ which is not W-chordal}
\end{figure}
Then $\C$ does not have any W-simplicial vertex, so is not W-chordal. But it is clear that $\C \in \mathfrak{C}_3$. Hence the class of W-chordal clutters is strictly contained in $\mathfrak{C}_d$.
\end{ex}

\subsection{Emtander's chordal class vs $\mathfrak{C}_d$}
Towards partial generalization of Fr\"oberg's theorem, E.~Emtander in \cite{Emtander} has also defined several concepts of chordality (called triangulated, triangulated*, chordal and having perfect elimination ordering) by different approaches and he showed that all of these concepts are the same \cite[Theorem-definition 2.1]{Emtander}. He also made a good discussion about different attempts on defining chordal clutters in section 2.2 of \cite{Emtander}. Next, he defined the notion of ``generalized chordal clutter'' as a generalization of his previous objects, celebrating by showing that such clutters admit  linear resolutions over any field $K$. In the following, we show that the class of generalized chordal clutters as defined by Emtander is strictly contained in $\mathfrak{C}_d$.

\begin{defn}[{\cite[Definition 4.1]{Emtander}}] \label{Definition of E-chordal}
A \textit{generalized chordal clutter} is a $d$-uniform clutter, obtained
inductively as follows:
\begin{itemize}
\item[(i)] $\C_{n,d}$ is a generalized chordal clutter for $n, d\in \mathbb{N}$;
\item[(ii)] If $\C$ is generalized chordal, then so is $\C \cup_{\C_{i,d}} \C_{n,d}$, for $0 \leq i < n$ (This we think of glueing $\C_{n,d}$ to $\C$ by identifying $\C_{n,d}$ with the corresponding part, $\C_{i,d}$ of $\C$);
\item[(iii)] If $\C$ is generalized chordal, $F \subset V(\C)$ with $|F|=d$ and there exists $e\subset F$, $|e|=d-1$, with $e \notin \mathrm{SC}(\C)$, then $\C \cup {F}$ is generalized chordal.
\end{itemize}
\end{defn}
Let us denote by `E-chordal', the class of generalized chordal clutters as defined by Emtender.

\begin{rem}
A $d$-uniform clutter $\C$ which is obtained inductively by (i)-(ii) of Definition~\ref{Definition of E-chordal}, is called chordal by Emtander while chordal in this paper, is an element of $\mathfrak{C}_d$. Emtander showed that, if $\C$ is an E-chordal clutter, then $I \left( \bar{\C} \right)$ has a linear resolution over any field (\cite[Theorem 4.1]{Emtander}). 
In particular, if $\C$ is a chordal clutter in Emtander's sense, then $I \left( \bar{\C} \right)$ has a linear resolution over any field. 
In his later work, he showed that for a chordal clutter $\C$, in fact, the ideal $I \left( \bar{\C} \right)$ has linear quotients \cite{Emtander2} while it is still an open question whether this is true for generalized chordal clutters as well.
\end{rem}

In the following, we show that, E-chordal clutters are strictly contained in $\mathfrak{C}_d$. To do this, first we prove the following lemma.

\begin{lem} \label{complete}
Let $v \in [n]$ and $T$ be the set $\left\{ e \subset [n] \colon \quad |e|=d-1, \; v \in e \right\}$. Then, by a suitable ordering of the elements of $T$, $T = \left\{ e_1, \ldots, e_m \right\}$, we have 
$$e_i \in \mathrm{Simp} \left( \left(\C_{n,d} \setminus e_1 \right) \setminus \cdots  \setminus e_{i-1}  \right),$$
for $2 \leq i \leq m$.
\end{lem}

\begin{proof}
Without loss of generality we may assume that $v=1$. We define a total order on $T$ as follows: 
\begin{itemize}
\item[] let $e=\{1, i_1, i_2, \ldots, i_{d-2}\}$ and $e'=\{1, j_1, j_2, \ldots, j_{d-2}\}$, where $1<i_1< i_2< \cdots< i_{d-2}$ and $1< j_1< j_2< \cdots< j_{d-2}$. Then $e\prec e'$ if and only if there exists an integer $t$ such that $i_1=j_1, \ldots, i_{t-1}=j_{t-1}$ and $i_t<j_t$.
\end{itemize}
Now, let $T=\{e_1, \ldots, e_m\}$ with $e_1 \prec e_2 \prec \cdots \prec e_m$. Clearly, $e_1=\{1,2, \ldots, d-1\}$, and since $\mathcal{C}_{n,d}$ is $d$-complete we have $N[e_1]=[n]$ which is a clique of $\mathcal{C}_{n,d}$ and so $e_{1} \in \mathrm{Simp}(\mathcal{C}_{n,d})$.

Set $\mathcal{D}_0:=\mathcal{C}_{n,d}$, and for $i \geq 1$, let $\mathcal{D}_i:=\mathcal{D}_{i-1} \setminus e_i$. Let $i>1$ and assume that $e_{i-1} \in \mathrm{Simp} \left( \mathcal{D}_{i-2} \right)$ and  $e_i=\left\{ 1, i_1, \ldots, i_{d-2} \right\} \in \mathrm{SC} \left( \mathcal{D}_{i-1} \right)$ with $1 < i_1 < \cdots < i_{d-2}$. We claim that $N_{\mathcal{D}_{i-1}}\left[ e_i \right]= e_i \cup \left\{i_{d-2}+1, \ldots, n \right\}$.  
\begin{itemize}
\item[] \textit{Proof of the claim.} To show that $\{i_{d-2}+1, \ldots, n\} \subseteq N_{\mathcal{D}_{i-1}} \left[ e_i \right]$, pick an element $j \in \left\{ i_{d-2}+1, \ldots, n \right\}$. We show that $e_i \cup \left\{ j \right\} \in \mathcal{D}_{i-1}$. It will follow that $j \in \mathrm{N}_{\mathcal{D}_{i-1}} \left[ e_i \right]$. Since $e_i \cup \left\{ j \right\} \in \mathcal{C}_{n,d}$, it is enough to prove that $e_k \not\subset e_i \cup \left\{ j \right\}$, for all $1 \leq k \leq i-1$.
Suppose that there exists $1 \leq k \leq m$, $k \neq i$, such that $e_k \subset e_i \cup \left\{j \right\}$. So, $e_k= \left( \left\{ j \right\} \cup e_i \right) \setminus \left\{ i_s \right\}$ for some $1 \leq s \leq d-2$. Since $j > i_l$, for all $1 \leq l \leq d-2$, we have $e_i \prec e_k$. So $k>i$, which implies that $e_k \not\subset e_i \cup \left\{ j \right\}$. Therefore, $e_i \cup \left\{j \right\}\in \mathcal{D}_{i-1}$. 

Conversely, suppose that $j \in N_{\mathcal{D}_{i-1}} \left[ e_i \right] \setminus e_i$. Then $e_i \cup \left\{ j \right\} \in \mathcal{D}_{i-1}$. Thus $e_k \not\subset e_i \cup \left\{ j \right\}$ for all $1 \leq k \leq i-1$. In particular, $e_k \neq \left( \left\{ j \right\} \cup e_i \right) \setminus \left\{ i_l \right\}$, for all $1 \leq l \leq d-2$. Since $i_l \neq 1$, we have $\left( \left\{ j \right\} \cup e_i \right) \setminus \left\{ i_l \right\} \in T$ and so for any $1 \leq l \leq d-2$ there exists $i+1 \leq k_l \leq m$ such that $e_{k_l}= \left( \left\{ j \right\} \cup e_i \right) \setminus \left\{ i_l \right\}$. Since $k_{d-2}>i$, the  order of elements of $T$ implies that  $e_i \prec e_{k_{d-2}}$, and so $j>i_{d-2}$. Hence $ N_{\mathcal{D}_{i-1}} \left[ e_i \right]\setminus e_i \subseteq \{i_{d-2}+1, \ldots, n\}$.
\end{itemize}
Now, we prove that $N_{\mathcal{D}_{i-1}} \left[ e_i \right]$ is a clique in $\mathcal{D}_{i-1}$. Let $F \subseteq N_{\mathcal{D}_{i-1}} \left[ e_i \right]$ with $|F|=d$. We show that $F \in \mathcal{D}_{i-1}$. Since $F \in \mathcal{C}_{n,d}$, it is enough to prove that $F$ does not contain any $e_k$ with $k<i$.

Assume that $e_k\subset F$ for some $k<i$. Since $e_k \subset N_{\mathcal{D}_{i-1}} \left[ e_i \right]$ and $k \neq i$, there exists  $j \in  \left\{ i_{d-2}+1, \ldots, n \right\}$ such that $j\in e_k$. Therefore, $j>i_l$ for all $i_l \in e_i$ and hence $e_i \prec e_k$ which implies that $k>i$. Thus $e_k \not\subset F$ for all $1 \leq k \leq i-1$, as desired. This completes the proof.
\end{proof}

\begin{cor} \label{complete clutters are chordal}
The complete clutter $\C_{n,d}$ is chordal.
\end{cor}

\begin{prop}[\textbf{E-chordals are chordal}] \label{E-chordals are chordal}
If $\C$ is a $d$-uniform E-chordal clutter, then $\C \in \mathfrak{C}_d$.
\end{prop}

\begin{proof}
The proof is recursive as in the definition of E-chordal clutters. First note that by Corollary~\ref{complete clutters are chordal}, the clutter $\C$ is chordal if it is a complete clutter. Assume that $\C = \mathcal{D} \cup {F}$, where  $\mathcal{D}$ is E-chordal and $F \subset V \left( \C \right)$ is such that $|F|=d$ and there exists $e \subset F$, $|e|=d-1$, with $e\notin \mathrm{SC} \left( \mathcal{D} \right)$. Since $F$ is the only circuit containing $e$, we conclude that, $e \in \mathrm{Simp} \left( \C \right)$. Moreover, $\C \setminus e =\mathcal{D}$. So induction on the number of circuits of $\C$ shows that $\C \in \mathfrak{C}_d$.

Suppose now that $\C = \mathcal{D} \cup_{\C_{i,d}} \C_{n,d}$. In this case there exists $v \in V \left( \C_{n,d} \right) \setminus V \left( \mathcal{D} \right)$. Without loss of generality we may suppose that $v=n$. Lemma~\ref{complete} implies that, with a suitable ordering of all submaximal circuits of $\C_{n,d}$ containing $v$, say $e_1, \dots, e_m$, we have $e_1 \in \mathrm{Simp} \left( \C_{n,d} \right)$ and  $e_j \in \mathrm{Simp} \left( \C_{n,d} \setminus e_1 \setminus \cdots \setminus e_{j-1} \right)$, for $2 \leq j \leq m$. Since $v \in V \left( \C_{n,d} \right) \setminus V \left( \mathcal{D} \right)$, we have $e_1 \in \mathrm{Simp} \left( \C \right)$ and  $e_j\in \mathrm{Simp} \left( \C \setminus e_1 \setminus \cdots \setminus e_{j-1} \right)$ for $2 \leq j \leq m$. Clearly, $\C \setminus e_1 \setminus \cdots \setminus e_{m} = \mathcal{D} \cup_{\C_{i,d}} \C_{(n-1),d}$ which is E-chordal. Induction, now completes the proof.
\end{proof}

\begin{ex}
The class of $d$-uniform E-chordal clutters are strictly contained in $\mathfrak{C}_d$. To see this, consider the following $3$-uniform clutter $\C$:
\[
\C= \left\{ 123, 124, 134, 235, 245, 345, 125, 135, 145 \right\}.
\]
Then, one may check that $\C$ is not E-chordal, while $\C \in \mathfrak{C}_3$. Also, it is worth to say that, the class of E-chordal and W-chordal are not contained in each other (see \cite[Example 4.8]{Woodroofe}). 
\end{ex}

\subsection{Other known chordalities}
Another notion of chordality, has been defined by Van Tuyl and Villarreal in \cite{VanTuyl-Villarreal}. They called a clutter $\C$ (not necessarily $d$-uniform) chordal, if every minor of $\C$ has a free vertex, that is, a vertex appearing in exactly one circuit of $\C$. Note that a free vertex is obviously W-simplicial and hence a clutter with free vertex property, is W-chordal. Let us denote the class of chordal clutters in the sense of Van Tuyl and  Villarreal, by VTV-chordal. It is worth to note that $\C$ is a VTV-chordal if and only if   $\C = \mathcal{F} (\Delta)$, where $\Delta$ is the clique complex of a chordal graph (\cite[Example 4.5]{Woodroofe}). Hence this class is a small subclass of W-chordal clutters.

Another notion of chordality can be found in \cite{Nevo}, where the authors defined the concept of `resolution $l$-chordal'  for a simplicial complex. Due to \cite[Section 3]{Nevo}, a simplicial complex $\Delta$ is resolution $l$-chordal, if $\tilde{H}_l \left( \Delta_W; K \right) =0$, for every subset $W$ of vertices of $\Delta$. Now, let $\Delta$ be a simplicial complex such that $I_\Delta$ is generated by elements of degree $d$. It is well-known and immediately concluded from Hochster's formula \cite[Theorem 5.1]{Hochster} that $I_\Delta$ has a $d$-linear resolution, if and only if $\tilde{H}_{i} \left( \Delta_W ; K \right) =0$, for every subset $W$ of vertices of $\Delta$ and for all $i \neq d-2$. With the notion as in \cite{Nevo}, this is equivalent to saying that $\Delta$ is resolution $l$-chordal for all $l \neq d-2$. In \cite[Theorem 5.1]{Nevo}, the authors refined this condition with saying that $\Delta$ is resolution $l$-chordal for all integers $l \in [d-1, 2d-3]$. Hence, chordality in \cite{Nevo} may be viewed as a refinement of the Hochster's formula.

Most of the attempts to generalize Fr\"oberg's theorem have been intended to extend the notion of chordal graphs to higher dimensions and to show that, with a new definition, the associated ideal has a linear resolution but not vice versa. But, in the work of Connon and Faridi \cite{ConnonFaridi}, the direction is in a different way. Indeed, the authors defined the concept of `chorded complexes' and showed that, if $I_{\Delta}$ has a linear resolution over any field $K$, then $\Delta$ is chorded \cite[Corollary 6.2]{ConnonFaridi}. Let us denote the chorded class as defined by Connon and Faridi, by CF-chordal. There are examples of CF-chordal complexes, whose associated ideals do not have linear resolution (see e.g. \cite[Example 7.2]{ConnonFaridi}), and obviously such examples do not belong to $\mathfrak{C}_d$, by Theorem~\ref{chordallin}.

\subsection{Conclusion}
Let us denote the class of $d$-uniform clutters whose circuit ideals have a linear resolution over any field by LinRes. By the discussions in this section, we obtain the following diagram:

\begin{center}
\begin{tabular}{r}
 VTV-chordal  $\subsetneq$  W-chordal\\
 E-chordal
\end{tabular}
$\bigg\} \subsetneq  \mathfrak{C}_d \subseteq \mbox{LinRes} \subsetneq \mbox{CF-chordal} $
\end{center}

A C++ program has been prepared to check whether a $d$-uniform clutter belongs to the class ${\mathfrak C}_d$.  Using this program, we checked some strange examples which have linear resolution over any field and we found that all are in the class ${\mathfrak C}_d$. Some of these examples are \cite[Examples 5 and 7]{Hashi} and \cite[Theorem 3.5]{Duval}. The source code of this program and the details of computations can be found in \cite{code}. Motivated by the above diagram and these evidences, we propose the following question:

\begin{question} \label{Characterization question}
Does there exist any $d$-uniform clutter $\C$ such that the ideal $I \left( \bar{\C} \right)$ has a linear resolution over any field, but $\C$ is not in the class $\mathfrak{C}_d$?
\end{question}

\subsection{Linear quotients}
Let $\C \neq \C_{n,d}$ be a $d$-uniform clutter and $I= I \left( \bar{\C} \right)$. If $\C  \in \mathfrak{C}_d$, then by Theorem~\ref{chordallin}, we know that $I$ has a $d$-linear resolution over any field. On the other hand, ideals with linear quotients which are generated in a same degree, have also linear resolutions over any field. So it is natural to ask, whether the ideal associated to a chordal clutter has linear quotients. In the following example we show that the ideals associated to chordal clutters are not contained in the class of ideals with linear quotients. However, it is known that, if $G$ is a chordal graph, then $I \left( \bar{G} \right)$ (and all of its powers) has linear quotients \cite[Corollary 3.2]{HH05}.

\begin{ex}
Let $\Delta$ be a triangulation of the dunce hat with $8$ vertices as shown in Figure~\ref{Dunce hat}, which is  originally introduced by Zeeman \cite{Zeeman}. Let $\C$ be the $5$-uniform clutter $\C= \C_{8, 5} \setminus \mathcal{F} \left( \bar{\Delta} \right)$, where $\bar{\Delta} = \langle [8] \setminus F \colon \quad F \in \mathcal{F} \left( \Delta \right) \rangle$.

\begin{figure}[!htp]
\begin{tikzpicture}[line cap=round,line join=round,>=triangle 45, scale=0.5]
\definecolor{pinky}{rgb}{0.6,0.2,0.}
\fill[color=pinky,fill=pinky,fill opacity=0.1] (8.,6.) -- (3.,-2.) -- (13.,-2.) -- cycle;
\fill[color=pinky,fill=pinky,fill opacity=0.1] (8.,2.) -- (7.,1.) -- (7.46,0.) -- (8.68,0.) -- (9.,1.) -- cycle;
\draw (8.,6.)-- (3.,-2.);
\draw (3.,-2.)-- (13.,-2.);
\draw (13.,-2.)-- (8.,6.);
\draw (8.,2.)-- (7.,1.);
\draw (7.,1.)-- (7.46,0.);
\draw (7.46,0.)-- (8.68,0.);
\draw (8.68,0.)-- (9.,1.);
\draw (9.,1.)-- (8.,2.);
\draw (8.,2.)-- (8.,6.);
\draw (7.,1.)-- (8.,6.);
\draw (7.,1.)-- (6.09,2.9);
\draw (7.,1.)-- (4.8,0.97);
\draw (7.46,0.)-- (3.,-2.);
\draw (4.85,0.97)-- (7.46,0.);
\draw (7.46,0.)-- (6.48,-2.);
\draw (8.68,0.)-- (6.48,-2.);
\draw (8.,2.)-- (7.46,0.);
\draw (8.,2.)-- (8.68,0.);
\draw (8.68,0.)-- (10.,-2.);
\draw (8.68,0.)-- (13.,-2.);
\draw (9.,1.)-- (13.,-2.);
\draw (9.,1.)-- (11.1,0.98);
\draw (9.,1.)-- (9.87,3.0);
\draw (8.,2.)-- (9.87,3.0);
\begin{scriptsize}
\draw [fill](8.,6.) circle (1.5pt);
\draw (8.0,6.3) node {$1$};
\draw [fill] (3.,-2.) circle (1.5pt);
\draw (3.0,-2.35) node {$1$};
\draw [fill] (13.,-2.) circle (1.5pt);
\draw (13.0,-2.35) node {$1$};
\draw [fill] (6.08,2.9) circle (1.5pt);
\draw (5.75,3.0) node {$3$};
\draw [fill] (4.85,0.97) circle (1.5pt);
\draw(4.55,1.07) node {$2$};
\draw [fill] (11.14,0.97) circle (1.5pt);
\draw(11.5,1.07) node {$2$};
\draw [fill] (9.87,3.0) circle (1.5pt);
\draw(10.2,3.1) node {$3$};
\draw [fill] (6.48,-2.) circle (1.5pt);
\draw(6.48,-2.4) node {$3$};
\draw [fill] (10.,-2.) circle (1.5pt);
\draw (10.0,-2.4) node {$2$};
\draw [fill] (8.,2.) circle (1.5pt);
\draw (7.73,2.15) node {$6$};
\draw [fill] (7.,1.) circle (1.5pt);
\draw(6.6,1.23) node {$5$};
\draw [fill] (7.46,0.) circle (1.5pt);
\draw(7.5,-0.4) node {$4$};
\draw [fill] (8.68,0.) circle (1.5pt);
\draw (8.65,-0.45) node {$8$};
\draw [fill] (9.,1.) circle (1.5pt);
\draw (9.4,1.23) node {$7$};
\end{scriptsize}
\end{tikzpicture}
\caption{A triangulation of the dunce hat} \label{Dunce hat}
\end{figure}
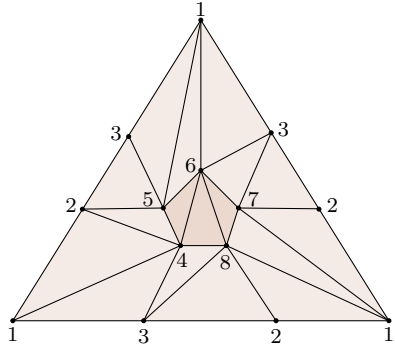
Then $\C$ contains $39$ circuits and $I \left( \bar{\C} \right) = I_{\Delta^\vee}$, the Stanley-Reisner ideal of the Alexander dual of $\Delta$. Since $\Delta$ is not shellable \cite[Section III, p. 84]{Stanley}, the ideal $I:=I \left( \bar{\C} \right)$ does not have linear quotients (\cite[Prop. 8.2.5]{HHBook}). But with the following order on simplicial submaximal circuits of $\C$, we get $\C \in \mathfrak{C}_5$ and hence it has a linear resolution over any field.
\[
\begin{array}{rrrrrrr}
& e_1 = 1237 & e_2= 1234 & e_3= 1235 & e_4= 1236 & e_5= 1467 & e_6= 1468 \\
& e_7 = 1246 & e_8= 1245 & e_9= 1247 & e_{10}= 1256 & e_{11}= 1257 & e_{12}= 1267\\
& e_{13} = 1346 & e_{14}= 1345 & e_{15}= 1347 & e_{16}= 1367 & e_{17}= 1356 & e_{18}= 1357\\
& e_{19} = 1457 & e_{20}= 1567 & e_{21}= 3567 & e_{22}= 2356 & e_{23}= 2578 & e_{24}= 2357 \\
& e_{25} = 2345 & e_{26}= 2347 & e_{27}= 2346 & e_{28}= 2367 & e_{29}= 2457 & e_{30}= 2456 \\
& e_{31} = 3457 & e_{32}= 3467 & e_{33}= 4567 & & &
\end{array}
\]
\end{ex}

\begin{question}
Find a subclass of chordal clutters such that their associated ideal have linear quotients.
\end{question}

\section*{acknowledgement}
The authors would like to express their deepest gratitude to J\"urgen Herzog, for lots of valuable discussions during writing this manuscript. The source code of the program~\cite{code} has been prepared by Iman Kiarazm. The authors also would like to thank him for spending a plenty of time for writing the codes.


\begin{thebibliography}{99}

\bibitem{code}
A Program for Detecting Chordality. available at: \href{http://iasbs.ac.ir/~yazdan/chordality.html}{\texttt{www.iasbs.ac.ir/$\sim$yazdan/chordality.html}}

\bibitem{Adiprasito}
K. A. Adiprasoto, \textit{Higher chordality II: Toric chordality via the McMullen--Weil Lefschetz Map}, preprint, (2015).  \href{http://arxiv.org/abs/1503.06640}{\texttt{arXiv:1503.06640}}

\bibitem{Nevo}
K. A. Adiprasoto E. Nevo, and J. A. Samperhigher, \textit{Higher chordality I. From graphs to complexes}, preprint, (2015). \href{http://arxiv.org/abs/1503.05620}{\texttt{arXiv:1503.05620}}

\bibitem{ConnonFaridi}
E.~Connon and S.~Faridi, \textit{Chorded complexes and a necessary condition for a monomial ideal to have a linear resolution}, J.~Combin.~Theory~Ser.~A,  \textbf {120},  no. 7, pp. 1714--1731, (2013).

\bibitem{Dirac}
G.~A.~Dirac, \textit{On rigid circuit graphs}, Abh. Math. Sem. Univ. Hamburg, \textbf{38}, pp. 71--76, (1961).

\bibitem{Duval}
A. M. Duval, B. Goeckner, C. J. Klivans and J. L. Martin, \textit{A non-partitionable Cohen-Macaulay simplicial complex}, preprint, (2015).  \href{http://arxiv.org/abs/1504.04279}{\texttt{ 	arXiv:1504.04279}}

\bibitem{Emtander}
E.~Emtander, \textit{A class of hypergraphs that generalizes chordal graphs}, Math.~Scand. \textbf{106}, no. 1, pp. 50--66, (2010).

\bibitem{Emtander2}
E. Emtander, F. Mohammadi, and S. Moradi, \textit{Some algebraic properties of hypergraphs}, Czechoslovak Mathematical Journal, vol \textbf{61}, no. 3, pp. 577--607, (2011).

\bibitem{Eagon-Reiner}
J. A. Eagon and V. Reiner, \textit{Resolutions of Stanley-Reisner rings and Alexander duality}, J. Pure and Applied Algebra \textbf{130}, pp. 265--275, (1998).

\bibitem{Fr}
R.~Fr\"{o}berg, \textit{On Stanley--Reisner rings}, in: Topics in algebra, Banach Center Publications, \textbf{26} Part 2, pp. 57--70, (1990).

\bibitem{HaVanTuyl}
H. T. H\`a and A. Van Tuyl, \textit{Monomial ideals, edge ideals of hypergraphs, and their graded Betti numbers}, J. Algebraic Combin. \textbf{27} (2), pp 215--245, (2008).

\bibitem{Hashi}
M. Hachimori, \textit{Decompositions of two-dimensional simplicial complexes}, Discrete Mathematics, \textbf{308}, Issue 11,  pp. 2307--2312, (2008).

\bibitem{HH05}
J. Herzog and T. Hibi, \textit{The depth of powers of an ideal}, J. Algebra, \textbf{291},  pp. 534--550, (2005).

\bibitem{HHBook}
J.~Herzog and T.~Hibi, \textit{Monomial Ideals}, in: GTM  \textbf{260}, Springer, London, (2011).

\bibitem{HHZ}
J.~Herzog and T.~Hibi and X.~Zheng, \textit{Monomial ideals whose powers have a linear resolution}, Math. Scand. \textbf{95} , pp. 23--32, (2004).

\bibitem{Hochster}
M. Hochster, \textit{Cohen-Macaulay rings, combinatorics, and simplicial complexes}, in: Ring Theory, II, Proc. Second Conf., Univ. Oklahoma, Norman, Okla., 1975, in: Lecture Notes in Pure and Appl. Math., vol. 26, Dekker, New York, 1977, pp. 171--223.

\bibitem{Katzman}
M. Katzman, \textit{Characteristic-independence of Betti numbers of graph ideals}, Journal of Combinatorial Theory, Series A, vol. \textbf{113} (3), pp 435--454, (2006).

\bibitem{MNYZ}
M.~Morales, A.~Nasrollah Nejad, A.~A.~Yazdan Pour and R.~Zaare-Nahandi, \textit{Monomial ideals with $3$-linear resolutions}, Annales de la Facult\'e des Sciences de Toulouse, S\'er. \textbf{6}, 23: (4), pp 877--891, (2014). \href{http://arxiv.org/abs/1207.1789}{\texttt{arXiv:1207.1790v1}}

\bibitem{MYZ2}
M. Morales, A. A. Yazdan Pour and R. Zaare-Nahandi, \textit{The regularity of edge ideals of graphs}, J. Pure and Applied Algebra, vol. \textbf{216}, issue 12, pp. 2714--2719 (2012).

\bibitem{MYZ}
M.~Morales, A.~A.~Yazdan Pour and R.~Zaare-Nahandi, \textit{Regularity and Free Resolution of Ideals which are Minimal to $d$-linearity}, to appear in Math, Scand. (2012). \href{http://arxiv.org/abs/1207.1790}{\texttt{arXiv:1207.1789v1}}

\bibitem{Sturmfels}
B. Sturmfels, \textit{Four counterexamples in combinatorial algebraic geometry}, J. Algebra \textbf{230}, 282--294, (2000).

\bibitem{Stanley}
R.P. Stanley, \textit{Combinatorics and Commutative Algebra}, second ed., Birkh\"auser, Boston, (1996).

\bibitem{VanTuyl-Villarreal}
A. Van Tuyl and R. H. Villarreal, \textit{Shellable graphs and sequentially Cohen-Macaulay bipartite graphs}. J. Combin. Theory Ser. A, \textbf{115}(5): pp 799--814, (2008).

\bibitem{Woodroofe}
R.~Woodroofe, \textit{Chordal and sequentially Cohen-Macaulay clutters}, Electron.~J.~Combin. {\textbf 18}, no. 1, Paper 208, 20 pages, (2011). \href{http://arxiv.org/abs/0911.4697}{\texttt{arXiv:0911.4697v4}}

\bibitem{Zeeman}
E. C. Zeeman, \textit{On the dunce hat}, Topology 2, 341--358, (1963).
\end{thebibliography}
\end{document}